\documentclass[11pt, reqno]{amsart}   	% use "amsart" instead of "article" for AMSLaTeX format
\usepackage{amssymb, latexsym}
\usepackage{amsmath}
\usepackage[toc,page]{appendix}
\usepackage{braket}
\usepackage{breqn}
\usepackage{caption}
\usepackage{color}
\usepackage{dsfont}
\usepackage{esint}
\usepackage[T1]{fontenc}
\usepackage{geometry}         
\usepackage{graphicx}
\usepackage{hyperref}
\usepackage{latexsym}
\usepackage{mathrsfs}
\usepackage{mathtools}
\usepackage{subcaption}
\usepackage{float}
\restylefloat{table}
\usepackage{listings}
\usepackage[dvipsnames]{xcolor} %can be used to highlight text

\usepackage{times}
\usepackage{url} %does nice formatting of urls
\usepackage{tikz}

\newcommand{\bburl}[1]{\textcolor{blue}{\url{#1}}}

\theoremstyle{plain}
\numberwithin{equation}{section}
\newtheorem{thm}{Theorem}[section]
\newtheorem{theorem}[thm]{Theorem}
\newtheorem{lemma}[thm]{Lemma}

\newtheorem{corollary}[thm]{Corollary}

\newtheorem{example}[thm]{Example}

\newtheorem{definition}[thm]{Definition}
\newtheorem{notation}[thm]{Notation}

\newtheorem{proposition}[thm]{Proposition}
\newtheorem{remark}[thm]{Remark}

\def\al{\alpha}

\def\be{\beta}

\def\ga{\gamma}

\def\ka{\kappa}

\def\pa{\partial}

\def\th{\theta}

\def\E{\mathbb{E}}

\def\N{\mathbb{N}}
\def\P{\mathbb{P}}

\def\R{\mathbb{R}}

\def\fa{\forall}

\def\til{\widetilde}

\newcommand\beq{\begin{equation}}
\newcommand\eeq{\end{equation}}
\newcommand\bea{\begin{eqnarray}}
\newcommand\eea{\end{eqnarray}}
\newcommand\bi{\begin{itemize}}
\newcommand\ei{\end{itemize}}
\newcommand\ben{\begin{enumerate}}
\newcommand\een{\end{enumerate}}

\everymath{\displaystyle}

\begin{document}
\title{On The Enumeration and Asymptotic Analysis of Fibonacci Compositions}
\author{Joshua M. Siktar}
\date{\today}
\maketitle
\centerline{\bf Abstract}
    We study Fibonacci compositions, which are compositions of natural numbers that only use Fibonacci numbers, in two different contexts. We first prove inequalities comparing the number of Fibonacci compositions to regular compositions where summands have a maximum possible value. Then, we consider asymptotic properties of Fibonacci compositions, comparing them to compositions whose terms come from positive linear recurrence sequences. Finally, we consider analogues of these results where we do not allow the use of a certain number of consecutive Fibonacci numbers starting from $F_2 = 1$.
\pagestyle{myheadings}

\tableofcontents
\section{Introduction and basic definitions}\label{intro}
The goal of this paper is to unify and extend the theory on compositions of Fibonacci numbers, employing both the more general theory of compositions and the theory of positive linear recurrence sequences. To state our objectives more precisely, we must first define the notions of composition and partition. 
\begin{definition}[Compositions and Partitions]\label{compPartDef}
A \textit{composition} of an integer $n$ is an ordered representation of positive integers $\{n_1, n_2, \dots, n_k\}$ for which
\begin{equation}\label{compositionDef}
n \ = \ n_1 + n_2 + \dots + n_k.
\end{equation}
On the other hand, a \textit{partition} of $n$ is a collection of non-increasing positive integers $m_1 \geq m_2 \geq \dots \geq m_j$ for which
\begin{equation}\label{partitionDef}
n \ = \ m_1 + m_2 + \dots + m_j.
\end{equation}
\end{definition}
From these definitions one can easily observe that the number of partitions of any positive integer $n$ is in general less than the number of its compositions. Indeed, any partition of $n$ is a composition of $n$, and any reordering of its summands gives rise to another composition of $n$. Also, both compositions and partitions can repeat summands. While there are well-known asymptotic results on the enumeration of partitions (see \cite{andrews1976theory, bona2006walk, Rob, Ste}), in this paper we focus on compositions. In the literature, compositions are studied in two settings that are clearly equivalent: the direct enumeration of solutions to \eqref{compositionDef}, and the placement of balls into bins that are arranged in a line. Both enumerative and asymptotic problems in this area have been studied extensively; see \cite{Beck, bona2006walk, Bond, Fla, Gra, Mos, Mpg, Mur, Ruz93, Ruz95, Sedi} and the references therein. 

\begin{definition}[Fibonacci Numbers]\label{FibDef}
We denote the Fibonacci numbers as the sequence where $F_0 = 0$, $F_1 = 1$, and 
\begin{equation}\label{FibRecurDef}
F_k \ = \ F_{k - 1} + F_{k - 2}
\end{equation}
for all $k \geq 2$. We also define a Fibonacci composition as a composition of a natural number obtained using only Fibonacci numbers as summands.
\end{definition}

We also rely on Zeckendorf decompositions to motivate the questions we study, so we define those next.

\begin{definition}[Zeckendorf decomposition]\label{ZeckDecompDef}
A \textit{Zeckendorf decomposition} is a partition of a natural number into a sum of positive, non-adjacent Fibonacci numbers, where the Fibonacci numbers are taken as in Definition \ref{FibDef}.
\end{definition}

It is well known that every positive integer has a unique Zeckendorf decomposition; see \cite{Ze} for a standard proof via induction and a greedy algorithm. The combinatorics and limiting behavior of Zeckendorf decompositions and related problems on Fibonacci numbers are studied across a growing body of papers, including \cite{Bec, Cat, Che, Chu, Fan, Fat, Fil, Gue, Heb, How, Kol, LiM, MilWan1}.

With that in mind, we turn our attention to the more scarcely studied compositions of Fibonacci numbers. Naturally, a \textit{Fibonacci composition} is a composition of a natural number using only Fibonacci numbers. One of the challenges of studying Fibonacci compositions is the difficulty in finding closed formulas for their number. This is very different from compositions that can use any of the natural numbers, or even compositions with lower and upper limits on what summands can be used. The recent work \cite{Bond} explores the enumerative aspects of this problem in great detail. Since we can not expect to replicate the same type of analysis with Fibonacci compositions\footnote{The OEIS entry \cite{Slo} gives some information on summation representations for Fibonacci compositions, however.}, we instead turn our attention to asymptotic results. As a starting point, the paper \cite{Kno} gives a generating function approach to finding the number of ways to split $n$ balls into bins so each bin contains a Fibonacci number of balls. Then, the paper \cite{Sedi} performs a generalization of this, for sets $S \subset \N^+$ beyond the Fibonacci numbers, by introducing the notion of an interpreter. Neither paper gives any attention to minimum capacity problems, where we require each summand to be sufficiently large, so we use this paper as an opportunity to generalize existing results in that direction.

One of the other major objects studied in this paper is the positive linear recurrence sequence, which we proceed to define carefully.
\begin{definition}[PLRS]\label{PLRSDef} A \textit{positive linear recurrence sequence} (PLRS) of order $L$ is a sequence of positive integers $\{H_j\}^{\infty}_{j = 1}$ with the following properties:
\begin{enumerate}
\item There are non-negative integers $L, c_1, \dots, c_L$ such that 
\begin{equation}\label{PLRSDefEq1}
H_{i + 1} \ = \ c_1H_i + c_2H_{i - 1} + \dots + c_LH_{i + 1 - L} 
\end{equation}
for all $i \geq L,$ where $L, c_1, c_L > 0$.
\item We have $H_1 = 1$, and the following decomposition holds for $1 \leq i < L$:
\begin{equation}\label{PLRSDefEq2}
   H_{i + 1} \ = \ c_1H_i + c_2H_{i - 1} + \dots + c_iH_1 + 1.
\end{equation}
\end{enumerate}
\end{definition}
These sequences are the subject of study of many recent papers in number theory; see for instance \cite{Bol, Gil, LiM, MilWan1, MilWan2}. Of course, the Fibonacci numbers are a PLRS with order $2$, where $c_1 = c_2 = 1$. 

Now we briefly outline the contents of the remainder of the paper. First, Section \ref{combiFib} provides estimates comparing the number of Fibonacci compositions to the number of standard compositions, including the case where the first $m - 1$ summands can not be used in the Fibonacci compositions. Beyond that section, the brief Section \ref{(0, 1)} provides the necessary background and results on the theory of power series. Section \ref{asymp} provides asymptotic results for increasing sequences of positive numbers, and Section \ref{PLRS} compares asymptotics of Fibonacci compositions to compositions involving summands from an arbitrary PLRS. Finally, Section \ref{end} discusses a few possible directions for future projects.

%%%%%%%%%%%
%%%%%%%%%%%
%%%%%%%%%%%

\section{Combinatorics of Fibonacci Compositions}\label{combiFib}

The combinatorial analysis in this section will be split into two parts. First we consider the special case where all Fibonacci numbers can be used in composition; the second subsection discusses the more general case where we exclude the first $m - 1$ Fibonacci numbers from being used. The estimates in this section are not necessarily the tightest possible, but seeking further refinements using the same techniques does not add much insight. To perform this analysis, we introduce some notation for this section.

\begin{notation}[Composition variables]\label{CompositionVAr}
For $m, n \in \N^+$, we denote: 
\begin{enumerate}
    \item $X_n$ as the random variable for the number of summands in a random Fibonacci composition of $n$;
    \item $X_{n, m}$ as the random variable for the number of summands in a random Fibonacci composition that does not use any summands in the list $\{F_1, F_2, \dots, F_{m - 1}\}$;
    \item $S_n$ as the total number of summands amongst all Fibonacci compositions of $n$;
    \item $\iota(n)$ as the index of the smallest Fibonacci number not less than $n$, i.e., $\iota(n) := \min\{k \in \N^+, F_k \geq n\}$.
\end{enumerate}
\end{notation}

%%%%%%%%%%%%

\subsection{Compositions without missing summands}\label{nomisssummands}
We first define the sets that will be analyzed in this subsection .
\begin{notation}\label{G_nU_n}
For $n \in \N^+$, we denote
\begin{equation}\label{G_nDef}
    \overline{G}_n \ := \ \Big\{\{a_i\}^{j}_{i = 1} \subset \N^+, \sum^{j}_{i = 2}a_i = n, a_i = F_i \ \text{for some} \ i\Big\}
\end{equation} 
as the set of Fibonacci compositions of $n$, and 
\begin{equation}\label{I_nDef}
    \overline{I}_{n, [k]} \ := \ \Big\{\{a_i\}^{j}_{i = 1} \subset \N^+, \sum^{j}_{i = 1}a_i = n, a_i \leq k\Big\}
\end{equation}
as the set of compositions of $n$ where each summand is at most $k$. Naturally we also denote $G_n := |\overline{G}_n|$ and $I_{n, [k]} := |\overline{I}_{n, [k]}|$
\end{notation}
Later in this section we will bound the size of a collection of Fibonacci compositions from above by the size of a collection of compositions where each summand has a maximum possible value. While it is trivial that $G_n \leq I_{n, [n]}$, we use this as inspiration to find a tighter bound of the same nature. 

\begin{theorem}\label{FibVsMaxCapac}
For $n \geq 6$, we have the inequality
\begin{equation}\label{FibVsMaxCapacIneq}
    G_n \ \leq \ I_{n, [\iota(n) - 1]}.
\end{equation}
\end{theorem}

\begin{proof}
The proof will use, in a non-trivial way, that $\iota(n) \geq 6$ if and only if $n \geq 6$. Our strategy is to construct a bijection between $\overline{G}_n$ and a subset of $\overline{I}_{n, [\iota(n) - 1]}$, which will verify \eqref{FibVsMaxCapacIneq}. We define the set 
\begin{equation}\label{G1Def}
\mathcal{G}_1 \ := \ \{\{a_i\}^{j}_{i = 1} \in \overline{G}_n, a_i \leq \iota(n) - 1 \ \fa i\},
\end{equation}and define 
\begin{equation}\label{G2Def}
    \mathcal{G}_2 \ := \ \overline{G}_n \setminus \mathcal{G}_1.
\end{equation} Precisely, the set $\mathcal{G}_2$ is the collection of all Fibonacci compositions of $n$ using at least one Fibonacci number taking value at least $\iota(n)$. With that in mind, we define our map $T$ on $\overline{G}_n$ as follows:
\begin{enumerate}
    \item If $\ell \in \mathcal{G}_1$, then $T(\ell) := \ell$.
    \item If $\ell \in \mathcal{G}_2$, then in place of every summand $a_i \geq 8$, insert the summands $\{4, 1, 1, \dots, 1, 4\}$, where exactly $a_i - 8$ consecutive $1$s are used. 
\end{enumerate}
Notice that $\mathcal{G}_1 \subseteq \overline{I}_{n, [\iota(n) - 1]}$, and for compositions belonging to $\mathcal{G}_2$, we have taken every Fibonacci summand with value at least $8$, and replaced it by a collection of $4$s and $1$s. Since $n \geq 6$, necessarily $\iota(n) \geq 6$, so $4$ is a legal summand for compositions in $\overline{I}_{n, [\iota(n) - 1]}$, and we conclude that $T(\overline{G}_n) \subset \overline{I}_{n, [\iota(n) - 1]}$.

Now we define the map $T^{-1}$ on the set of compositions $T(\overline{G}_n)$ as follows:
\begin{enumerate}
    \item If $\ka \in \mathcal{G}_1$ then define $T^{-1}(\ka) := \ka$.
    \item If $\ka \in T(\overline{G}_n)\setminus\mathcal{G}_1$, then read the composition from left to right. Between the first $4$ and the second $4$, add those $4$s, and all summands in between them; replace all of those summands by their sum. Repeat this process until there are no more $4$s in the composition.
\end{enumerate}
Due to the construction of the set $T(\overline{G}_n)$, any composition in the set will contain an even number of $4s$, so the procedure we describe is well-defined. In particular, since compositions in $\overline{G}_n$ only contain Fibonacci numbers, the only $4$s that appear are from splitting a larger summand into smaller ones.

Finally, to show that $T$ is a bijection, the only non-trivial step is demonstrating that $T^{-1}(T(\ell)) = \ell$ whenever $\ell \in \mathcal{G}_2$. In applying $T$ to $\ell$, we look for summands that are at least $8$, and replace each of them with a $4$, then some number of $1$s, and then another $4$. All other summands are left unchanged (including $5$s, since $5 < \iota(n)$). Reversing this process is straightforward, as we just replace the aforementioned string of $4$s and $1$s by their sum. The proof is complete.
\end{proof}

\begin{remark}\label{FibIntIneqRmk}
We chose the parameters for this inequality very carefully; $4$ is the smallest natural number that is not a Fibonacci number, and $5$ is the only Fibonacci number larger than $4$ that can not be written as a sum of two $4$s and some [possibly empty] set of other natural numbers; thus we insisted $\iota(n)$ be large enough to ensure that $5$ was a legal summand.
\end{remark}

In addition to this analysis, we can quickly calculate a summation formula for $I_{n, [\iota(n) - 1]}$ using Theorem 8 from \cite{Bond}. Keeping the notation from that paper, let $R_{n, \ell, k}$ denote the number of compositions of $n$ with $\ell$ summands, each summand taking value at most $k$. Then we have the formula
\begin{equation}\label{BallsInBinsEq1}
    R_{n, \ell, k} \ = \ \sum^{\ell}_{t = 0}(-1)^t{{\ell}\choose{t}}{{n - t(k + 1) + \ell - 1}\choose{\ell - 1}}.
\end{equation}
We may set $k := \iota(n) - 1$, and then the inequality $\Bigl\lfloor{\frac{n}{\iota(n) - 1}\Bigr\rfloor} \leq \ell \leq n$ indicates how many summands we can have. Summing \eqref{BallsInBinsEq1} over all suitable values of $\ell$, we conclude that
\begin{equation}\label{BallsInBinsEq2}
    I_{n, [\iota(n) - 1]} \ = \ \sum^{n}_{\ell = \lfloor{\frac{n}{\iota(n) - 1}\rfloor}}\sum^{\ell}_{t = 0}(-1)^t{{\ell}\choose{t}}{{n - t(k + 1) + \ell - 1}\choose{\ell - 1}}.
\end{equation}
We can use this enumeration to analyze events whose underlying probability space is all Fibonacci compositions of $n$, and then we can compute a lower bound on the expected value of this random variable.
\begin{theorem}\label{EXnLB}
For $n \geq 6$, we have the lower bound
\begin{equation}\label{EXnLBIneq}
    \E[X_n] \ \geq \ \frac{nF_n}{2I_{n, [\iota(n) - 1]}}.
\end{equation}
\end{theorem}

\begin{proof}
Since $X_n$ is a non-negative random variable, we may apply Markov's inequality (see for instance \cite[Theorem 18.21]{Juk}) to obtain the inequality
\begin{equation}\label{EXnLBIneq1}
    \frac{2}{n}\E[X_n] \ \geq \ \P\left[X_n \geq \frac{n}{2}\right].
\end{equation}
Moreover, we observe that if a given Fibonacci composition of $n$ only has $1$s and $2$s as summands, then it has at least $\frac{n}{2}$ summands, but the converse is not necessarily true. As an example, a Fibonacci composition of $8$ is $8 = 3 + 1 + 1 + 1 + 1 + 1$, which has more than $4$ summands but includes a $3$ as one of those summands.

It follows that
\begin{equation}\label{EXnLBIneq2}
\P\left[X_n \geq \frac{n}{2}\right] \ \geq \ \P[Z_{[1, 2]}],
\end{equation}
where $Z_{[1, 2]}$ is the event where a Fibonacci composition only has $1$s and $2$s as summands. It is well-known that the number of ways to compose $n$ using only $1$s and $2$s is equal to $F_n$, so
\begin{equation}\label{EXnLBIneq3}
    \P\left[X_n \geq \frac{n}{2}\right] \ \geq \ \frac{F_n}{G_n},
\end{equation}
and combining \eqref{EXnLBIneq1} and \eqref{EXnLBIneq3} gives us
\begin{equation}\label{EXnLBIneq4}
    \E[X_n] \ \geq \ \frac{nF_n}{2G_n}.
\end{equation}
Finally, thanks to the inequality \eqref{FibVsMaxCapacIneq} we find the lower bound \eqref{EXnLBIneq}.
\end{proof}

By taking an alternative route starting from \eqref{EXnLBIneq4}, we get another inequality of interest in its own right.
\begin{corollary}\label{EXnLBCor}
For all $n \geq 6$, we have
\begin{equation}\label{EXnLBCorIneq}
    S_n \ \geq \ \frac{nF_n}{2}.
\end{equation}
\end{corollary}

\begin{proof}
Write $\E[X_n]$ as $\frac{S_n}{G_n}$. We may interpret $\E[X_n]$ in this way because we are determining the average number of summands in a Fibonacci composition of $n$.
\end{proof}

%%%%%%%%%%%%%%%%%%%%
%%%%%%%%%%%%%%%%%%%%
%%%%%%%%%%%%%%%%%%%%
%%%%%%%%%%%%%%%%%%%%
%%%%%%%%%%%%%%%%%%%%

\subsection{Compositions with missing summands}\label{missingsumandsbounds}

Now we seek results analogous to those in the previous subsection, except now we do not allow the use of the first $m - 1$ Fibonacci numbers as summands. We need to ensure that such Fibonacci compositions can even exist for all numbers sufficiently large, no matter how many of the first Fibonacci numbers we do not allow. It turns out we can even find an exact threshold beyond which this behavior always takes place.

\begin{lemma}\label{cyclinglemma}
For all $m \in \N^+$, there exists $N(m) \in \N_0$ such that for all $n \geq N$, there exists a sequence of coefficients $\{\be_j\}^{\infty}_{j = m} \subset \N_0$ such that $n = \sum^{\infty}_{j = m}\be_jF_j$. Furthermore, we can choose $N(m) := F_mF_{m + 1}$.
\end{lemma}

\begin{proof}
Considering Definition \ref{FibRecurDef} of the Fibonacci numbers, we will show that for all $n \geq F_mF_{m + 1}$, there exist $\th_m, \th_{m + 1} \in \N_0$ such that $n = \th_mF_m + \th_{m + 1}F_{m + 1}$. 

Clearly, if $n = N(m) = F_mF_{m + 1}$, we can set $\th_m := F_{m + 1}$ and $\th_{m + 1} := 0$. From here, we will proceed via induction, using $n = N(m)$ as the base case. Suppose a suitable composition holds for some $n \geq N(m)$, and we want to find one for $n + 1$. The strategy will be to add $1$ to the existing composition in such a way that the coefficients do not become negative. Since any two consecutive Fibonacci numbers are relatively prime, we may write the integer $1$ as a linear combination of any two consecutive Fibonacci numbers. In fact, there exist coefficients $\eta_m, \eta_{m + 1}, \tau_m, \tau_{m + 1}$ such that
\begin{equation}\label{cyclinglemmaEq1}
    \begin{aligned}
    \eta_mF_m + \eta_{m + 1}F_{m + 1} \ &= \ 1 \\
    \tau_mF_m + \tau_{m + 1}F_{m + 1} \ &= \ 1,
    \end{aligned}
\end{equation}
where without loss of generality, we assume $\eta_m \geq 0$, $\eta_{m + 1} \leq 0$, $\tau_m \leq 0$, and $\tau_{m + 1} \geq 0$.

We will want to add one of these equations to the existing composition $n = \th_mF_m + \th_{m + 1}F_{m + 1}$; which one we choose depends on the values of the coefficients. The following two Fibonacci identities, known as the \textit{Cassini identities}, are well-known for $m \in \N^+$ and can be proven by induction:
\begin{equation}\label{FibIDcycling1}
    F_m^2 - F_{m - 1}F_{m + 1} \ = \ (-1)^{m + 1};
\end{equation}
\begin{equation}\label{FibIDcycling2}
    F_{m - 1}F_m - F_{m - 2}F_{m + 1} \ = \ (-1)^m.
\end{equation}
These identities tell us if $m$ is even, we can set $\eta_m := -F_m$, $\eta_{m + 1} := F_{m - 1}$, $\tau_m := F_{m - 1}$, and $\tau_{m + 1} := -F_{m - 2}$; likewise, if $m$ is odd, then we can set $\eta_m := F_m$, $\eta_{m + 1} := -F_{m - 1}$, $\tau_m := -F_{m - 1}$, and $\tau_{m + 1} := F_{m - 2}$. Then, if either of the following pairs of statements is true, we have found a valid composition for $n + 1$:
\begin{enumerate}
    \item $\th_m + \eta_m \geq 0$ and $\th_{m + 1} + \eta_{m + 1} \geq 0$;
    \item $\th_m + \tau_m \geq 0$ and $\th_{m + 1} + \tau_{m + 1} \geq 0$.
\end{enumerate}
Assume first that $m$ is even, and then $\th_{m + 1} + \eta_{m + 1} \geq 0$ and $\th_m + \tau_m \geq 0$ are both immediate. We assume for sake of contradiction that $\th_{m + 1} < -\tau_{m + 1}$ and that $\th_m < -\eta_m$. In this case, we may use the specific values of $\eta_{m + 1}$ and $\tau_m$ to obtain the inequality chain
\begin{equation}\label{cyclinglemmaEq2}
    N \ \leq \ \th_mF_m + \th_{m + 1}F_{m + 1} \ < \ -\eta_mF_m - \tau_{m + 1}F_{m + 1} \ = \ F_m^2 + F_{m - 2}F_{m + 1} \ = \ F_mF_{m + 1} - 1 \ < \ N,
\end{equation}
which is a contradiction. The case where $m$ is odd is similar.
\end{proof}

\begin{remark}\label{countingTwoSummandComps}
From this construction it is easy to see how many such compositions are possible using just $F_m$ and $F_{m + 1}$ when $n \geq F_mF_{m + 1}$. If $n = \ell F_mF_{m + 1}$ for some $\ell \in \N^+$, then we can construct a composition of $n$ using just $F_m$ and $F_{m + 1}$ as summands in $\ell + 1$ ways. If $\ell F_mF_{m + 1} < n < (\ell + 1)F_mF_{m + 1}$, then we can write such a composition in exactly $\ell$ ways. 
\end{remark}

Now, we provide a [more delicate] variant of Theorem \ref{FibVsMaxCapac} that holds when we consider Fibonacci compositions without the summands $\{F_1, F_2, \dots, F_{m - 1}\}$; the challenge is identifying what number can fill the role of the $4$s in the proof of Theorem \ref{FibVsMaxCapac}. The notation used is as follows: 
\begin{enumerate}
    \item $G_{n\setminus\{F_1, \dots, F_{m - 1}\}}$ refers to the number of Fibonacci compositions of $n$ that do not use $\{F_1, \dots, F_{m - 1}\}$ as summands;
    \item $I_{n, \Big\{m, m + 1, \dots, \left\lfloor\frac{5n}{8}\right\rfloor\Big\}}$ refers to the number of compositions of $n$ that only use $\Big\{m, m + 1, \dots, \left\lfloor\frac{5n}{8}\right\rfloor\Big\}$ as possible summands.
\end{enumerate}

\begin{theorem}\label{missingSummandsGToI} We have the inequality
\begin{equation}\label{missingSummandsGToIIneq}
    G_{n\setminus\{F_1, \dots, F_{m - 1}\}} \ \leq \ I_{n, \Big\{m, m + 1, \dots, \left\lfloor\frac{5n}{8}\right\rfloor\Big\}}
\end{equation}
whenever $m \geq 4$, and either of the following sets of conditions are met:
\begin{enumerate}
    \item $m$ is not a Fibonacci number and $\frac{5}{8}n \geq 3m + 1$;
    \item $m + 1$ is not a Fibonacci number and $\frac{5}{8}n \geq 3m + 3$.
\end{enumerate}
\end{theorem}

\begin{proof}
Allowing summands in our compositions up to $\frac{5}{8}n$ ensures that for any $m \geq 4$, we have $\frac{F_{m + 1}}{F_m} \geq \frac{8}{5}$. Moreover, between $\frac{5}{8}n$ and $n$ there is at most one Fibonacci number, and we denote it as $F_k$. Furthermore, since $F_k > \frac{n}{2}$, it can only appear in a composition of $n$ once. This simplifies the forthcoming analysis.

Denote $\overline{G}_{n\setminus\{F_1, \dots, F_{m - 1}\}}$ as the set of Fibonacci compositions without the summands $\{F_1, F_2, \dots, F_{m - 1}\}$, and $\overline{I}_{n, \Big\{m, m + 1, \dots, \left\lfloor\frac{5n}{8}\right\rfloor\Big\}}$ as the set of compositions of $n$ that only use $\Big\{m, m + 1, \dots, \left\lfloor\frac{5n}{8}\right\rfloor\Big\}$ as possible summands. Notice that as long as $n \geq F_mF_{m + 1}$, Lemma \ref{cyclinglemma} assures us that the set $\overline{G}_{n\setminus\{F_1, \dots, F_{m - 1}\}}$ is nonempty; if this set is empty, then the theorem trivially holds.

Just as in the proof of Theorem \ref{FibVsMaxCapac}, we will split the set $\overline{G}_{n\setminus\{F_1, \dots, F_{m - 1}\}}$ into two sub-collections. The first of these is 
\begin{equation}\label{mathcalG1Def}
\mathcal{G}_1 \ := \ \Bigg\{\{a_i\}^{\ell}_{i = 1} \in \overline{G}_{n\setminus\{F_1, \dots, F_{m - 1}\}}, a_i \leq \left\lfloor\frac{5n}{8}\right\rfloor\Bigg\}.
\end{equation}
To accompany \eqref{mathcalG1Def}, we also define
\begin{equation}\label{mathcalG2Def}\mathcal{G}_2 \ := \ \overline{G}_{n\setminus\{F_1, \dots, F_{m - 1}\}}\setminus\mathcal{G}_1.
\end{equation}
Precisely, the set $\mathcal{G}_2$ is the collection of all Fibonacci compositions of $n$ using $F_k$ as a summand, without using any of $\{F_1, \dots, F_{m - 1}\}$ as summands. With that in mind, we define our map $T$ on $\overline{G}_{n\setminus\{F_1, \dots, F_{m - 1}\}}$ as follows:
\begin{enumerate}
    \item If $\ell \in \mathcal{G}_1$, then define $T(\ell) := \ell$.
    \item If $\ell \in \mathcal{G}_2$ and $m$ is not a Fibonacci number, then in place of $F_k$, insert the summands $\{m, m + 1, m + 1, \dots, m + 1, r, m\}$. We use $j_1$ copies of $m + 1$, with $j_1$ solving the equation $F_k - 2m = j_1(m + 1) + r_1$, where $r_1 \in \{m + 2, m + 3, \dots, 2m + 1\}$.
    \item If $\ell \in \mathcal{G}_2$ and $m$ is a Fibonacci number, then in place of $F_k$, insert the summands $\{m + 1, m + 2, m + 2, \dots, m + 2, r, m + 1\}$. We use $j_2$ copies of $m + 2$, with $j_2$ solving the equation $F_k - 2(m + 1) = j_2(m + 2) + r_2$, where $r_2 \in \{m + 3, m + 4, \dots, 2m + 3\}$
\end{enumerate}
Some remarks are needed here. First notice that since $m \geq 4$, $m$ and $m + 1$ can not both be Fibonacci numbers. Moreover, the bounds relating $n$ and $m$ assure the existence solutions $(j_1, r_1)$ and $(j_2, r_2)$ to the equations $F_k - 2m = j_1(m + 1) + r_1$ (if $m$ is not a Fibonacci number) and $F_k - 2(m + 1) = j_2(m + 2) + r_2$ (if $m$ is a Fibonacci number) respectively. Indeed, in the case where $m$ is not a Fibonacci number, the bounds $\frac{5}{8}n \geq 3m + 1$ and $F_k > \frac{5}{8}n$ tell us we may write $F_k$ as a sum of two $m$s, some $j_1$ copies of $m + 1$, and a ``leftover" quantity $r_1$ that is definitely not equal to $1$, $2$, $3$, or $5$. The case where $m$ is a Fibonacci number is similar.
Finally, it is clear from the definition of $T$ that 
\begin{equation}\label{TGnInclusion}
    T(\overline{G}_{n\setminus\{F_1, \dots, F_{m - 1}\}}) \ \subset \ \overline{I}_{n, \{m, m + 1, \dots, \lfloor{\frac{5n}{8}\rfloor}\}}.
\end{equation} 

Now we can define the map $T^{-1}$ on the set of compositions $T(\overline{G}_{n\setminus\{F_1, \dots, F_{m - 1}\}})$ as follows:
\begin{enumerate}
    \item If $\ka \in \mathcal{G}_1$ then define $T^{-1}(\ka) := \ka$.
    \item If $\ka \in T(\overline{G}_{n\setminus\{F_1, \dots, F_{m - 1}\}})$ and $m$ is not a Fibonacci number, then read the composition from left to right. Add all the terms between two consecutive appearances of $m$, including the two $m$s in the sum. Repeat this procedure until there are no more summands equal to $m$ in the composition.
    \item If $\ka \in T(\overline{G}_{n\setminus\{F_1, \dots, F_{m - 1}\}})$ and $m$ is a Fibonacci number, then do the same steps as in the previous bullet point, but looking for $m + 1$ instead of $m$.
\end{enumerate}
Due to the construction of the set $T(\overline{G}_{n\setminus\{F_1, \dots, F_{m - 1}\}})$, any composition in that set will contain an even number of $m$s (or $m + 1$s, when $m$ is a Fibonacci number), so the procedure we described is well-defined. The fact that $T$ is a bijection between $\overline{G}_{n\setminus\{F_1, \dots, F_{m - 1}\}}$ and $T(\overline{G}_{n\setminus\{F_1, \dots, F_{m - 1}\}})$ follows from the same reasoning used in the proof of Theorem \ref{FibVsMaxCapac}.
\end{proof}

\begin{remark}\label{FibVsMaxCapacRmk1}
These bounds are far from being tight. Still, we can control the number of Fibonacci compositions from above by some number of standard compositions, which are a better studied object, even when we do not allow the use of the first few Fibonacci numbers. This bound could be refined by removing or adjusting some of the simplifying assumptions, but it is not worth complicating the main ideas of the proof.
\end{remark}

Now we provide an analogue of Theorem \ref{EXnLB}, where we exclude $\{F_1, F_2, \dots, F_{m - 1}\}$ as possible summands. 

\begin{theorem}\label{EXnmLB}
We have the lower bounds
\begin{equation}\label{EXnmLBIneqA}
    \E[X_{n, m}] \ \geq \ \frac{n}{F_{m + 1}}\left\lfloor\frac{n}{F_mF_{m + 1}}\right\rfloor\frac{1}{G_{n, \setminus\{F_1, \dots, F_{m - 1}\}}}
\end{equation}
and
\begin{equation}\label{EXnmLBIneqB}
    \E[X_{n, m}] \ \geq \ \frac{n}{F_{m + 1}}\left\lfloor\frac{n}{F_mF_{m + 1}}\right\rfloor\frac{1}{I_{n, \{m, m + 1, \dots, \left\lfloor\frac{5n}{8}\right\rfloor\}}}
\end{equation}
whenever $m \geq 2$ and $n \geq F_mF_{m + 1}$.
\end{theorem}

\begin{proof}
Since $X_{n, m}$ is a non-negative random variable, we may apply Markov's Inequality to obtain the inequality
\begin{equation}\label{EXnmLBIneq1}
    \frac{F_{m + 1}}{n}\E[X_{n, m}] \ \geq \ \P\left[X_{n, m} \geq \frac{n}{F_{m + 1}}\right].
\end{equation}
If a given Fibonacci composition of $n$ only has $F_m$ and $F_{m + 1}$ as summands (which is guaranteed to exist when $n \geq F_m F_{m + 1}$ by Lemma \ref{cyclinglemma}), then it has at least $\frac{n}{F_{m + 1}}$ summands, but the converse is not necessarily true. It follows that
\begin{equation}\label{EXnmLBIneq2}
    \P\left[X_{n, m} \geq \frac{n}{F_{m + 1}}\right] \ \geq \ \P[Z_{n, [F_m, F_{m + 1}]}],
\end{equation}
where $Z_{n, [F_m, F_{m + 1}]}$ denotes the event where a Fibonacci composition of $n$ only has $F_m$ and $F_{m + 1}$ as summands. Due to Remark \ref{countingTwoSummandComps}, we have that
\begin{equation}\label{EXnmLBIneq3}
    \P[Z_{n, [F_m, F_{m + 1}]}] \ \geq \left\lfloor\frac{n}{F_mF_{m + 1}}\right\rfloor\frac{1}{G_{n, \setminus\{F_1, \dots, F_{m - 1}\}}}.
\end{equation}
Combining the inequalities \eqref{EXnmLBIneq1}, \eqref{EXnmLBIneq2}, and \eqref{EXnmLBIneq3} immediately yields \eqref{EXnmLBIneqA}. We may also utilize \eqref{missingSummandsGToIIneq} to deduce \eqref{EXnmLBIneqB}.
\end{proof}

%%%%%%%%%%
%%%%%%%%%%
%%%%%%%%%%
%%%%%%%%%%
\section{Asymptotic Analysis of Compositions}\label{asymp} 

We now complement our estimates pertaining to Fibonacci compositions of a fixed $n$ with an asymptotic analysis of these compositions as $n \rightarrow \infty$. Central to this asymptotic theory is the behavior of certain power series on the interval $(0, 1)$, so our first subsection reviews this theory, while the second introduces and proves some of our main asymptotic results.

\subsection{Power Series on $(0, 1)$}\label{(0, 1)} 

We begin this subsection by introducing notation for our power series.

\begin{notation}\label{BjSum}
Let $\{B_j\}^{\infty}_{j = 1}$ be a strictly increasing sequence of positive integers, and for each $m \in \N^+$, we define the power series 
\begin{equation}\label{THMGF} 
T_{B_m}(x) \ := \ \sum^{\infty}_{i = m}x^{B_i}.
\end{equation}
We also let $t_{B_m}(n)$ denote the number of compositions of $n$ using summands from $\{B_j\}^{\infty}_{j = 1}$, excluding the numbers $\{B_1, B_2, \dots, B_{m - 1}\}$. 
\end{notation}

Notice that we need to assume $\{B_j\}^{\infty}_{j = 1}$ is strictly increasing to assure that no powers are repeated in the summation; moreover, we indicate that the sum represents the generating function for choosing summands for a composition from a given increasing sequence of positive integers. It is then easily seen that the power series \eqref{THMGF} has a radius of convergence at least $1$ by comparison to the standard geometric series; we will henceforth restrict such power series to the interval $(0, 1)$, and consider real-valued roots of the power series on that interval.

We can freely take derivatives of $T_{B_m}$ in the interval $(0, 1)$, via term-by-term differentiation. Seeing that $T_{B_m}'(x) > 0$ for all $x \in (0, 1)$, and $T_{B_m}(0) = -1$, each equation of the form
\begin{equation}\label{BHMGFA}
    T_{B_m}(x) \ = \ 1
\end{equation}
has a unique root (where we've also used the blowup of $T_{B_m}$ at $x = 1$). Denote this sequence of roots as $\{\ga_{B_m}\}^{\infty}_{m = 1}$. We may define \eqref{THMGF} for any increasing sequence of natural numbers, including the Fibonacci numbers, though we will replace $T_{B_m}$ with $R_m$ when referring to this special case. 
\begin{notation}\label{FjSum}
We define the power series
\begin{equation}\label{BHMGF}
    R_m(x) \ := \ \sum^{\infty}_{i = m}x^{F_i},
\end{equation}
and let $r_m(n)$ denote the number of Fibonacci compositions of $n$ that do not use any of the summands $\{F_1, F_2, \dots, F_{m - 1}\}$. Each equation of the form 
\begin{equation}\label{BHMGFR}
    R_m(x) \ = \ 1
\end{equation}
will have a unique root, and we denote this sequence of roots as $\{\al_m\}^{\infty}_{m = 1}$.
\end{notation}

With these constructions in place, we have the following proposition, which may be regarded as a routine exercise in real analysis. For sake of completeness we provide a short proof.

\begin{proposition}\label{etaConv}
The sequence $\{\ga_{B_m}\}^{\infty}_{m = 1}$ converges to $1$ as $m \rightarrow \infty$.
\end{proposition}

\begin{proof}
By the Monotone Convergence Theorem for real numbers, if we show that the bounded sequence is strictly increasing, then it has a limit. To prove that the sequence is strictly increasing, notice that each function $R_m(x)$ is increasing on $(0, 1)$, and $f_{m + 1}(x)$ only differs from $R_m(x)$ in that it has one less power of $x$ in its expansion. Since $x^{B_i}$ is an increasing monomial on $(0, 1)$ for each $i \in \N^+$, necessarily $\ga_{B_{m + 1}} > \ga_{B_m}$. Thus the sequence $\{\ga_{B_m}\}^{\infty}_{m = 1}$ has a limit, which we will call $\ga_B$. Our construction of the sequence assures that $\ga_B \leq 1$, so we assume for sake of contradiction that $\ga_B < 1$.

Since $f_m$ is increasing on $(0, 1)$ and $\ga_B \geq \ga_{B_m}$ for each $m \in \N^+$, 
\begin{equation}\label{etaConvEq1}
    0 \ < \ -1 + \sum^{\infty}_{i = m}\ga_B^{B_i} \ < \ \infty,
\end{equation}
from which it follows that
\begin{equation}\label{etaConvEq2}
    \sum^{\infty}_{i = m}\ga_B^{B_i} \ > \ 1
\end{equation}
for all $m \in \N^+$. However, the series $\sum^{\infty}_{i = m}x^{B_i}$ absolutely converges for all $x \in (0, 1)$ by comparison with the geometric series $\sum^{\infty}_{i = 1}x^i$. That is, $\lim_{m \rightarrow \infty}\sum^{\infty}_{i = m}\ga_B^{B_i} = 0$, posing a contradiction. 
\end{proof}

The sequence of roots $\{\ga_{B_m}\}^{\infty}_{m = 1}$ appears in the results of the next section, where we start proving asymptotic results for compositions that are only comprised of positive integers belonging to a given sequence. Different terms from these sequences of roots will appear depending on how many terms from the beginning of $\{B_j\}^{\infty}_{j = 1}$ we preclude from use in compositions.

%%%%%%%%%%%%%
%%%%%%%%%%%%%
%%%%%%%%%%%%%

\subsection{Asymptotic Results}\label{asympRes}

Now we consider some results on the asymptotics of compositions for a given increasing sequence of positive integers. The starting point of this discussion is \cite[Theorem 7]{Kno}, proven in 2003, but we provide some generalizations and a new direction for what can be done with this type of result. Working with general increasing sequences of positive integers gives us the flexibility to later apply our theory to positive linear recurrence sequences.

Now we state the aforementioned \cite[Theorem 7]{Kno} on the asymptotics of Fibonacci compositions.

\begin{theorem}\label{FibKno} Let $r_2(n)$ denote the number of Fibonacci compositions of $n$ (with any number of summands), and let $R_2(x) := \sum^{\infty}_{i = 2}x^{F_i}$. Then, as $n \rightarrow \infty$, we have
\begin{equation}\label{fib2asymp}
    r_2(n) \ \sim \ \frac{1}{R_2'(\al_2)}\al_2^{-n - 1},
\end{equation}
where $\al_2$ is the unique positive root of $R_2(x) = 1$, namely $\al_2 \approx 0.5276126$, so that $r_2'(\al_2) \approx 3.3749752$. In addition, if $\overline{r_2}(n)$ denotes the mean number of summands in a random Fibonacci composition of $n$, then as $n \rightarrow \infty$, we have
\begin{equation}\label{fibasympPrime}
    \overline{r_2}(n) \ \sim \ \frac{1}{\al_2 R_2'(\al_2)}n \ \approx \ 0.561583n.
\end{equation}
\end{theorem}

Using the same method, and replacing $R_2(x)$ with $T_{B_m}(x)$, we obtain the following new theorem.

\begin{theorem}\label{FibAsympRemoveHead}
As $n \rightarrow \infty$ we have the asymptotic behavior
\begin{equation}\label{fib2asympjGen}
    {t_B}_m(n) \ \sim \ \frac{1}{{T_B}_m'(\ga_{B_m})}\ga_{B_m}^{-n - 1},
\end{equation}
where $\ga_{B_m}$ is the unique positive root of the equation $T_{B_m}(x) = 1$. In addition, if $\overline{t_B}_m(n)$ denotes the mean number of summands in a random Fibonacci decomposition of $n$, then as $n \rightarrow \infty$, we have 
\begin{equation}\label{fib2asympPrimejGen}
    \overline{{t_B}_m}(n) \ \sim \ \frac{1}{\ga_{B_m} {T_B}_m'(\ga_{B_m})}n.
\end{equation}
\end{theorem}

\begin{proof} We recall that the power series \eqref{THMGF} is analytic in $(0, 1)$. Calculating its first derivative yields 
\begin{equation}\label{eTheoremEq3}
  T_{B_m}'(x) \ = \ \sum^{\infty}_{i = m}B_ix^{B_i - 1}.
\end{equation} 
If we restrict $T_{B_m}$ to the real interval $[0, 1)$, we see that
\begin{equation}\label{eTheoremEq4}
    \lim_{s \rightarrow 1^+}T_{B_m}(s) \ = \ +\infty,
\end{equation}
so the Intermediate Value Theorem assures there exists some $\ga_{B_m} \in (0, 1)$ for which $T_{B_m}(\ga_{B_m}) = 1$. Moreover, $T_{B_m}'(x) > 0$ for all $x \in (0, 1)$ due to formula \eqref{eTheoremEq3}, so this choice of $\ga_H$ is unique. The function
\begin{equation}\label{eTheoremEq5}
    U_{H_m}(x) \ := \ \frac{1}{1 - T_{B_m}(x)}
\end{equation}
has a pole at $x = \ga_{B_m}$. Denote $g_k(n)$ as the number of ways to write $n$ as a composition of exactly $k$ summands from the list $\{B_j\}^{\infty}_{j = m}$. Then we recall that the local expansion of $F_k(x) := \sum^{\infty}_{i = 0}g_k(i)x^i$ around this pole (denoted $\al$) can be given as 
\begin{equation}\label{eTheoremEq6}
    F_k(x) \ \sim \ \frac{1}{\al T_{B_m}'(\ga_{B_m})} \cdot \frac{1}{1 - \frac{x}{\ga_{B_m}}},
\end{equation}
and then
\begin{equation}\label{eTheoremEq7}
    g_k(n) \ = \ [x^n]U_{H_m}(x) \ \sim \ \frac{1}{\ga_{B_m} T_{B_m}'(\ga_{B_m})}\ga_{B_m}^{-n} \ = \ \frac{1}{T_{B_m}'(\ga_{B_m})}\ga_{B_m}^{-n - 1}
\end{equation}
as $n \rightarrow \infty$, which proves \eqref{fib2asympjGen}.

It remains to prove \eqref{fib2asympPrimejGen}. To do so, we begin by defining $u$ as the number of summands in a given composition of $n$, which has the associated bi-variate generating function
\begin{equation}\label{eTheoremEq8}
    H_m(x, u) \ = \ \frac{1}{1 - uT_{B_m}(x)}.
\end{equation}
Since $\overline{t_{B_m}}(n)$ can be thought of the first [discrete] moment of $t_{B_m}(n)$, we can use the approach of differentiating identities to calculate $\overline{t_{B_m}}(n)$:
\begin{equation}\label{eTheoremEq9}
    \overline{t_{B_m}}(n) \ = \ \frac{1}{t_{B_m}(n)}[x^n]\frac{\pa}{\pa u}H_m(x, u) \ = \ \frac{1}{t_{B_m}(n)}\cdot\frac{T_{B_m}(x)}{(1 - T_{B_m}(x))^2}.
\end{equation}
With this in mind, we define
\begin{equation}\label{eTheoremEq10}
    j_{B_m}(x) \ := \ \frac{T_{B_m}(x)}{(1 - T_{B_m}(x))^2},
\end{equation}
and observe that $T_{B_m}$ has a pole of order $2$ at $x = \ga_{B_m}$. Thus the first term in the local expansion of $j_{B_m}(x)$ around $x = \ga_{B_m}$ is of order
\begin{equation}\label{eTheoremEq11}
    j_{B_m}(x) \ \sim \ \left(\frac{1}{T_{B_m}'(\al)}\right)^2\frac{1}{(x - \ga_{B_m})^2}.
\end{equation}
Extracting the coefficients from \eqref{eTheoremEq11} gives the asymptotic expansion
\begin{equation}\label{eTheoremEq12}
    [x^n]j_{B_m}(x) \ \sim \ \left(\frac{n}{\ga_{B_m}^2(T_{B_m}'(\ga_{B_m}))^2}\right)\ga_{B_m}^{-n}
\end{equation}
as $n \rightarrow \infty$. Plugging this into \eqref{eTheoremEq9} gives the formula \eqref{fib2asympPrimejGen}.
\end{proof}

The case of $m := 3$ for the Fibonacci numbers in Theorem \ref{FibAsympRemoveHead} is of special interest since the frequency of $1$s in compositions of Fibonacci numbers was studied in \cite{Kno}. In other words, it gives us a sense of how much flexibility we lose in building Fibonacci compositions when we can no longer use $1$ as a summand, when compared to Theorem \ref{FibKno}. 

\begin{theorem}[Theorem 8 in \cite{Kno}]\label{FibCompDensity1s}
The proportion of summands equal to $1$ in a random Fibonacci composition of $n$ tends to $\al_2 \approx 0.5276125$ as $n \rightarrow \infty$.
\end{theorem}

The related application of Theorem \ref{FibAsympRemoveHead} is as follows.

\begin{example}\label{FibAsympRemoveHeadj=3}
Let $r_3(n)$ denote the number of Fibonacci compositions of $n$ (with any number of summands) where the summand $1$ is not allowed, and let $R_3(x) := \sum^{\infty}_{i = 3}x^{F_i}$. As $n \rightarrow \infty$ we have
\begin{equation}\label{fib2asympr_3}
    r_3(n) \ \sim \ \frac{1}{R_3'(\al_3)}\al_3^{-n - 1},
\end{equation}
where $\al_3$ is the unique positive root of $R_3(x) = 1$, namely $\al_3 \approx 0.6855205$, so that $R_3'(\al_3) \approx 4.6054074$. In addition, if $\overline{r}_3(n)$ denotes the mean number of summands in a random Fibonacci composition of $n$ not using $1$ as a summand, then as $n \rightarrow \infty$, we have
\begin{equation}\label{fib2asympPrime}
    \overline{r_3}(n) \ \sim \ \frac{1}{\al_3 R_3'(\al_3)}n \ \approx \ 0.3167463n.
\end{equation}
\end{example}

Now we extend beyond the case $m = 3$, by providing some data concerning the sequence $\{\al_m\}^{\infty}_{m = 2}$; among other things this supports Proposition \ref{etaConv}. Table \ref{FibRemoveFirstData} keeps track of the values of the roots that appear between $0$ and $1$ as we change the value of $m$ (i.e., it keeps track of $\al_m$). As argued in the proof of Theorem \ref{FibAsympRemoveHead}, each equation in this family will have a unique root in the interval $(0, 1)$; moreover, the exact asymptotic behaviors of the number of compositions and the average number of summands in said compositions rely on the values of these roots. Finally, the last column of the table tracks the coefficient of $n$ in the asymptotic result for the mean number of summands in a random Fibonacci decomposition of $n$ where $\{F_1, F_2, \dots, F_{m - 1}\}$ are excluded. The first row in the table, for $m = 2$, represents what happens when every distinct positive Fibonacci number is allowed.

\begin{table}[h!]
\centering
\begin{tabular}{||c c c c||} 
 \hline
 Value of $m$ & $\substack{\text{Smallest Fibonacci}\\ \text{number allowed}}$ & Value of $\al_m$ & $\frac{1}{\al_mR_m'(\al_m)}$ \\ [0.5ex] 
 \hline\hline
 2 & 1& 0.5276126 & 0.5615856\\ 
 3 & 2 & 0.6855205& 0.3167463\\
 4 &3 & 0.7889604& 0.2018247 \\
 5 & 5& 0.8645115 &0.1232169 \\
 6 & 8& 0.9137569 & 0.0765024\\ 
  7 & 13& 0.9458315&0.0471977 \\ 
 8 & 21 & 0.9661554 &0.0291894\\
 9 &34 & 0.9789482& 0.0180354\\
 10 &55 & 0.9869358 & 0.0111476\\
 11 & 89& 0.9919058& 0.0068893\\ 
 12 & 144& 0.9949897 & 0.0042579\\ 
 13 & 233 & 0.9969005 & 0.0026315\\
 14 & 377& 0.9980833& 0.0016264 \\
 15 & 610 & 0.9988150& 0.0010051\\
 16 &987 & 0.9992674& 0.0006212\\ 
 17 & 1597& 0.9995472&0.0003839 \\ 
 18 & 2584 & 0.9997201& 0.0002373\\
 19 & 4181& 0.9998270& 0.0001466 \\
 20 & 6765 & 0.9998931& 0.0000906\\
 [1ex] 
 \hline
\end{tabular}
\caption{In this table we list, for different values of $m$, the approximated values of the roots $\al_m \in (0, 1)$ of $R_m(x) = 1$, and of the ratios $\frac{\overline{r_m}(n)}{n}$ when $n \rightarrow \infty$. The GitHub repository \url{https://github.com/jmsiktar/FibBins} contains the Python code used to produce the data in this table.}
\label{FibRemoveFirstData}
\end{table}

%%%%%%%%%%%%%%
%%%%%%%%%%%%%%
%%%%%%%%%%%%%%

\section{Comparing Fibonacci Compositions to General PLRS Compositions}\label{PLRS}
We look to compare the Fibonacci numbers to other positive linear recurrence sequences. 

%%%%%%%%%%%%%%
%%%%%%%%%%%%%%
%%%%%%%%%%%%%%

\subsection{Relative Asymptotics of PLRS Compositions}\label{relativeasymptotic}

Since every PLRS is a strictly increasing sequence of positive integers, Proposition \ref{etaConv} will apply to all such sequences. The central factor of understanding the relative asymptotics of compositions coming from Fibonacci numbers compared to some other PLRS is the asymptotic growth rate of such a sequence.

As a motivation for the forthcoming results, we showcase the relationship between the Fibonacci numbers and the Golden Ratio in an asymptotic framework. We may write
\begin{equation}\label{FibSetupEq1}
    \phi \ = \ \lim_{m \rightarrow \infty}\frac{F_{m + 1}}{F_m} \ = \ \lim_{m \rightarrow \infty}\frac{F_m + F_{m - 1}}{F_m} \ = \ 1 + \lim_{m \rightarrow \infty}\frac{F_{m - 1}}{F_m} \ = \ 1 + \frac{1}{\phi}.
\end{equation}
This is an equation that can be solved for $\phi$ to obtain the classical value of the Golden Ratio, $\frac{1 + \sqrt{5}}{2}$. Notice this is the only positive, real root of \eqref{FibSetupEq1}. This procedure of generating a polynomial where a root is the sequence's growth rate can be generalized to all PLRS, and we will need this tool in the present section. Consider a generic PLRS of order $L$ with coefficients $\{c_1, \dots, c_L\}$. Since $c_1, c_L > 0$, repeating \eqref{FibSetupEq1} in this general case gives
\begin{eqnarray}\label{GenSetupEq1}
\begin{aligned}
    &\be_H \ =: \ \lim_{m \rightarrow \infty}\frac{H_{m + 1}}{H_m} \ = \ \lim_{m \rightarrow \infty}\frac{\sum^{L}_{i = 1}c_iH_{m + 1 - i}}{H_m} \ = \\
    &c_1 + \sum^{L}_{i = 2}\frac{c_iH_{m + 1 - i}}{H_m} \ = \ \sum^{L}_{i = 1}\frac{c_i}{\be_H^{i - 1}},
    \end{aligned}
\end{eqnarray}
and we can multiply the resulting equation by $\be_H^{L - 1}$ to obtain the polynomial
\begin{equation}\label{GenSetupEq2}
    \be_H^L - \sum^{L}_{i = 1}c_i\be_H^{L - i}\ = \ 0.
\end{equation}
By Descartes' Rule of Signs, the polynomial \eqref{GenSetupEq2} will have exactly one positive root. Moreover, since $c_L > 0$, we can not have $0$ as a root of \eqref{GenSetupEq2}.

\begin{definition}[Asymptotic Growth Rate]\label{asympGrowthRate}
For a given PLRS $\{H_j\}^{\infty}_{j = 1}$, we define $\be_H := \lim_{m \rightarrow \infty}\frac{H_{m + 1}}{H_m}$ as the \textit{asymptotic growth rate} of $\{H_j\}^{\infty}_{j = 1}$. 
\end{definition}
The derivation of \eqref{GenSetupEq1} and \eqref{GenSetupEq2} assures that $\be_H$ is uniquely defined and strictly positive, for any PLRS.

\begin{definition}[Growth Polynomial]\label{growthPolyDef}
For a given PLRS $\{H_j\}^{\infty}_{j = 1}$, we call the left-hand side of \eqref{GenSetupEq2} the \textit{growth polynomial} associated with $\{H_j\}^{\infty}_{j = 1}$.
\end{definition}

Another simple and very useful property is the following: if $\{B_j\}^{\infty}_{j = 1}$ and $\{\til{H_m}\}^{\infty}_{m = 1}$ are two PLRS of the same order $L$ with coefficients $\{c_1, c_2, \dots, c_L\}$ and $\{\til{c_1}, \til{c_2}, \dots, \til{c_L}\}$ where $\til{c_i} \geq c_i$ for all $i \in \{1, 2, \dots, L\}$, then $\til{H}_j \geq H_j$ for all $j \in \N^+$; crucial to this property taking place is that all PLRS have initial term $1$. We will use this in the proof of Proposition \ref{FibPLRSFixM}, which relates compositions using only Fibonacci numbers to compositions using only numbers belonging to some other PLRS. First, however, we prove that if sufficiently many initial values of a PLRS are given, then the coefficients can be uniquely determined. This result will be needed for some of the technical lemmas in this section.

\begin{proposition}[Uniqueness from Initial Values]\label{uniquenessDetPLRS}
For any given $L + 1$ initial values of a PLRS (including $H_1 = 1$), there is at most one PLRS of length $L$ that satisfies those values.
\end{proposition}

\begin{proof}
Since $H_1 = 1$ by definition, the PLRS recurrence gives a system of $L$ equations used to determine $H_2, H_3, \dots, H_{L + 1}$. The resulting system will be lower triangular (if the equations are written starting with $H_2 = c_1H_1 + 1$ and ending with $H_{L + 1} = c_1H_L + \dots + c_LH_1$). In fact, the main diagonal will not have any zeros since all entries of a PLRS are strictly positive, so the coefficient matrix is invertible. However, the solution may have negative or fractional coefficients, so there is no guarantee that the solution corresponds to a PLRS.
\end{proof}

Now, we will fix $m \in \N^+$, and this number indicates how many of the first numbers in the sequences are not allowed to be summands in compositions. Recall from the statement of Theorem \ref{FibAsympRemoveHead} that $T_{B_m}(n)$ denotes the number of compositions of $n$ using the terms $\{H_m, H_{m + 1}, H_{m + 2}, \dots\}$ from any increasing sequence $\{B_j\}^{\infty}_{j = 1} \subset \N^+$. Meanwhile, $r_m(n)$ denotes this same quantity specifically for the Fibonacci numbers. Our next proposition tells us when the number of Fibonacci compositions asymptotically grows faster than compositions with summands from some other PLRS, and vice-versa.

\begin{proposition}[Fibonacci vs. PLRS Rates]\label{FibPLRSFixM} Let $\{H_j\}^{\infty}_{j = 1}$ be a PLRS according to Definition \ref{PLRSDef}, with asymptotic growth rate $\be_H$. If $m$ is sufficiently large, then the limit $\lim_{n \rightarrow \infty}\frac{t_{H_m}(n)}{r_m(n)}$ has the following properties:
\begin{enumerate}
    \item It equals $0$ if $\be_H > \phi$;
    \item It equals $+\infty$ if $\be_H < \phi$.
\end{enumerate}
\end{proposition}

\begin{proof}
Fix $m \geq 1$. If we apply Theorem \ref{FibAsympRemoveHead} for both the PLRS $\{H_j\}^{\infty}_{j = 1}$ and the Fibonacci numbers, we see that
\begin{equation}\label{FibPLRSFixMEq1}
    \lim_{n \rightarrow \infty}\frac{t_{H_m}(n)}{r_m(n)} \ = \ \lim_{n \rightarrow \infty}\frac{\frac{1}{T'_{H_m}(\ga_{B_m})} \ga^{-n - 1}_{H_m} }{\frac{1}{R'_m(\al_m)}\al_m^{-n - 1}} \ = \ \frac{R_m'(\al_m)}{T'_{H_m}(\ga_{B_m})}\lim_{n \rightarrow \infty}\left(\frac{\al_m}{\ga_{B_m}} \right)^{n + 1}.
\end{equation}
In other words, the relative behavior of $\al_m$ and $\ga_{B_m}$ dictate the asymptotic behavior of this quotient. This relative behavior is in turn determined by whether $\be_H < \phi$, $\be_H > \phi$, or $\be_H = \phi$. 

If $\be_H > \phi$, then the powers appearing in the power series $T_{B_m}(x)$ (except possibly for the first few) are larger than those appearing in the power series $R_m(x)$. Then $T_{B_m}(x) < R_m(x)$ for $x \in (0, 1)$, and both power series are increasing on this interval, so $\ga_{B_m} > \al_m$. The case of $\be_H < \phi$ is similar. 
\end{proof}

There are two limitations of this result. The first is that some PLRS are such that the first few values of the sequence are smaller than the Fibonacci numbers, but then $\{H_j\}^{\infty}_{j = 1}$ becomes larger; an example of this is the PLRS of order $6$ with coefficients $c_1 = 1$, $c_2 = c_3 = c_4 = c_5 = 0$, and $c_6 = 20$. This is the reason why we need $m$ to be sufficiently large, and leave the issue of small $m$ open for PLRS where this behavior takes place. On the other hand, if either of the following is true, we may set $m = 1$ in Proposition \ref{FibPLRSFixM}:
\begin{enumerate}
    \item $H_m \geq F_m$ for all $m \in \N^+$;
    \item $H_m \leq F_m$ for all $m \in \N^+$.
\end{enumerate}
The second caveat is that we need information about the value of $\be_H$ to determine which case of Proposition \ref{FibPLRSFixM} a given PLRS falls under. For a given PLRS, let $m_0$ denote the smallest index for which $F_{m_0} \neq H_{m_0}$. We will now study some particular PLRS for which we can obtain the needed information and apply Proposition \ref{FibPLRSFixM} readily. 

Moreover, it is worth noting that if $\be_H = \phi$ then we need information on which sequence has larger terms at the beginning of the sequence; otherwise the method used to analyze the other cases is inconclusive. Thus, one may ask if this is a case that ever occurs. The next proposition and example are intended to address this question.

\begin{proposition}\label{PLRSNOOrder2_3}
The only PLRS of order $2$ with asymptotic growth rate equal to $\phi$ is the Fibonacci numbers $\{F_j\}^{\infty}_{j = 1}$. Furthermore, there is no PLRS of order $3$ with asymptotic growth rate equal to $\phi$.
\end{proposition}

\begin{proof}
First we consider the case of PLRS of order $2$. Recall that the asymptotic growth rate of a PLRS is the unique positive root of the sequence's growth polynomial. Based on Definition \ref{growthPolyDef}, the coefficients of a growth polynomial are all rational. If the PLRS is of order $2$, this means the other root of the growth polynomial must be $\frac{1 - \sqrt{5}}{2}$. Furthermore, since every growth polynomial has leading coefficient $1$, the polynomial in question must be $x^2 - x - 1$, which is the growth polynomial for $\{F_j\}^{\infty}_{j = 1}$.

Now we consider the case of a PLRS of order $3$. Assume for sake of contradiction that there exists a PLRS of order $3$, where the positive root of the associated growth polynomial is $\phi$. Just as in the case of an order $2$ PLRS, we must have $\frac{1 - \sqrt{5}}{2}$ as another root of the polynomial. Now, there will be one more root of this growth polynomial, which we will call $a$. Since the coefficients of a growth polynomial are all real-valued, and the other roots are real-valued, $a \in \R$. In particular, we cannot have a complex conjugate pair of roots because a cubic polynomial has at most three distinct roots, and two of them are determined to be real. Furthermore, $\phi$ is the only positive root, and if the coefficients of the PLRS are $c_1$, $c_2$, and $c_3$, then $c_3 \neq 0$, so $0$ cannot be a root of the growth polynomial. Thus we have $a < 0$, and we observe that our growth polynomial must take the following form:
\begin{eqnarray}\label{PLRSNOOrder2_3Eq1}
\begin{aligned}
    (x^2 - x - 1)(x - a) \ &= \ 0 \Rightarrow \\
    x^3 + (-a - 1)x^2 + (a - 1)x + a \ &= \ 0.
\end{aligned}
\end{eqnarray}
In general for a PLRS, the first and last coefficients must be negative integers, and all other coefficients must be non-positive integers (recall Definition \ref{PLRSDef}). If $-a - 1 < 0$, then $a$ is a non-negative integer; on the other hand, formula \eqref{GenSetupEq2} implies that $0 < c_3 = -a$, which means that $a$ must be a strictly negative integer. We have reached a contradiction, thus no PLRS of order $3$ satisfying our conditions can exist.
\end{proof}

Meanwhile, in spite of Proposition \ref{PLRSNOOrder2_3}, there are other PLRS besides the Fibonacci numbers that yield an asymptotic growth rate of $\phi$. 

\begin{example}\label{PLRSOrder4CEX}
The simplest example of a PLRS of order greater than $2$ with asymptotic growth rate of $\phi$ is the PLRS of order $4$ with $c_1 = 1$, $c_2 = 0$, $c_3 = 1$, and $c_4 = 1$. This PLRS has the growth polynomial $\be_H^4 - \be_H^3 - \be_H - 1$. Its roots are $\phi$, $\frac{1 - \sqrt{5}}{2}$, $i$, and $-i$.
\end{example}

\begin{remark}\label{PLRSRootBd}
There are numerous results giving lower and upper bounds on the modulus of roots of a polynomial, based on the polynomial's coefficients. Such estimates appear in sources such as \cite{Cau, Sun}, but any such bound can give a sufficient condition for the growth rate of a particular PLRS being smaller or greater than $\phi$. In particular, many results are only written as an upper bound. However, if $a$ is an upper bound for the modulus of any root of the equation
\begin{equation}\label{growthMultInverse}
    -\sum^{L}_{i = 1}c_ix^i + 1 \ = \ 0,
\end{equation}
then $\frac{1}{a}$ is a lower bound for the modulus of any root of the equation
\begin{equation}\label{growthMultInverse2}
    x^L - \sum^{L}_{i = 1}c_ix^{L - i} \ = \ 0,
\end{equation}
because these two polynomials have roots which are multiplicative inverses. In other words, if $r_1, \dots, r_L$ are the roots of \eqref{growthMultInverse}, then the roots of \eqref{growthMultInverse2} are $\frac{1}{r_1}, \dots, \frac{1}{r_L}$. Notice also that \eqref{growthMultInverse2} is just \eqref{GenSetupEq2} with a different name for the variable.
\end{remark}

%%%%%%%%%%%%%%
%%%%%%%%%%%%%%
%%%%%%%%%%%%%%

\subsection{Case $m_0 \geq L + 2$}\label{m0Big}
Treating the previous subsection as the backbone of our asymptotic analysis, specifically Proposition \ref{FibPLRSFixM}, we consider separately, in the present and next two subsections, the cases $m_0 \geq L + 2$; $m_0 = L + 1$; and $m_0 \leq L$. Each case yields different scenarios for the possible relationships between the Fibonacci sequence $\{F_m\}^{\infty}_{m = 1}$ and an otherwise arbitrary PLRS $\{H_m\}^{\infty}_{m = 1}$.

This identity is needed to fully analyze the $m_0 = L + 2$ case when $L$ is odd (the case of $L$ being even is more straightforward).
\begin{lemma}\label{FibIdentityI}
If $L \geq 3$ is odd and $k \geq L + 1$, we have
\begin{equation}\label{FibIdentityIEq}
    2F_{k - L + 2} + \sum^{\frac{L - 1}{2} - 1}_{i = 1}F_{k - 2i + 1} \ = \ F_k + F_{k - L}.
\end{equation}
\end{lemma}

\begin{proof}
We will proceed by induction on $L$; the base case is to prove that
\begin{equation}\label{FibIdentityIEq1}
    2F_{k - 1} \ = \ F_k + F_{k - 3}
\end{equation}
for all $k \geq 4$, but this follows immediately from the Fibonacci recurrence. Now we may assume for some arbitrary $L \geq 3$ odd that \eqref{FibIdentityIEq} holds for all $k \geq L + 1$, and to complete the induction we must show that 
\begin{equation}\label{FibIdentityIEq2}
    2F_{k - L} + \sum^{\frac{L + 1}{2} - 1}_{i = 1}F_{k - 2i 1} \ = \ F_k + F_{k - L + 2}
\end{equation}
holds for all $k \geq L + 3$. If we rearrange the terms in \eqref{FibIdentityIEq} and \eqref{FibIdentityIEq2}, we observe that it suffices to prove that
\begin{equation}\label{FibIdentityIEq3}
2F_{k - L} + \sum^{\frac{L + 1}{2} - 1}_{i = 1}F_{k - 2i + 1} - F_{k - L - 2} \ = \ 2F_{k - L + 2} + \sum^{\frac{L - 1}{2} - 1}_{i = 1}F_{k - 2i + 1} - F_{k - L}.
\end{equation}
Subtracting away terms that appear on both sides of \eqref{FibIdentityIEq3} gives
\begin{equation}\label{FibIdentityIEq4}
    3F_{k - L} - F_{k - L - 2} \ = \ F_{k - L + 2},
\end{equation}
which is an immediate consequence of the Fibonacci recurrence. This completes the induction.
\end{proof}

%m0 is first index where sequences don't agree
\begin{lemma}\label{L+2=m0}
If $m_0 = L + 2$, then the following behaviors hold:
\begin{itemize}
    \item If $L$ is even then $F_{m + 1} = H_m$ for all $m \in \N^+$
    \item If $L$ is odd then $F_{m + 1} < H_m$ for all $m \geq m_0$.
\end{itemize}
\end{lemma}

\begin{proof}
Consider first the case where $L$ is even. Then $F_{m + 1} = H_m$ for all $m \in \{1, 2, \dots, L + 1\}$. One can easily check that a possible set of coefficients for the PLRS $\{H_j\}^{\infty}_{j = 1}$ is $c_i = 1$ for all odd $i$, and $c_L = 1$ (with all other coefficients being zero). Due to Proposition \ref{uniquenessDetPLRS}, this is the only possible set of coefficients for the PLRS. Now we can check that $F_{m + 1} = H_m$ for all $m \in \N^+$ via strong induction. Our construction of the sequence $\{H_j\}^{\infty}_{j = 1}$ automatically assures that $L + 1$ base cases hold. Let $k \geq L + 1$ and assume that $F_{m + 1} = H_m$ for all $m \leq k$, and we will show that $F_{k + 2} = H_{k + 1}$. This is demonstrated by the coefficients of $\{H_j\}^{\infty}_{j = 1}$ and the recurrence for the Fibonacci numbers:
\begin{eqnarray}\label{L+2=m0EqA1}
\begin{aligned}
H_{k + 1} \ &= \ H_{k - L + 1} + \sum^{\frac{L}{2} - 1}_{i = 0}H_{k - 2i} 
 \ = \ F_{k - L + 2} + \sum^{\frac{L}{2} - 1}_{i = 0}F_{k - 2i + 1} \\
\ &= \ F_{k - L + 4} + \sum^{\frac{L}{2} - 1}_{i = 1}F_{k - 2i + 1} 
\ = \ \dots 
\ = \ F_{k + 2},
\end{aligned}
\end{eqnarray}
completing the proof when $L$ is even.

Now consider the case where $L$ is odd. We will need a slightly weaker version of Lemma \ref{FibIdentityI}; precisely, we will use the inequality
\begin{equation}\label{L+2=m0EqB1}
    2F_{k - L + 2} + \sum^{\frac{L - 1}{2} - 1}_{i = 1}F_{k - 2i + 1} \ \geq \ F_k.
\end{equation}
By construction, $H_m = F_{m + 1}$ for all $m \in \{1, 2, \dots, L + 1\}$. One can easily check that a possible set of coefficients for the PLRS $\{H_j\}^{\infty}_{j = 1}$ is $c_i = 1$ for all odd $i$ except $L$, and $c_L = 2$ (with all other coefficients being zero). Again invoking Proposition \ref{uniquenessDetPLRS}, this is the only possible set of coefficients for the PLRS.

Let $k \geq L + 2$ and assume that $H_m > F_{m + 1}$ for all $m \in \{L + 2, L + 3, \dots, k\}$. Then to complete the proof, we must show that $H_{k + 1} > F_{k + 2}$. Due to our choice of coefficients, we have the recurrence
\begin{equation}\label{L+2=m0EqB2}
    H_{k + 1} \ = \ 2H_{k - L + 1} + \sum^{\frac{L - 1}{2} - 1}_{i = 0}H_{k - 2i}.
\end{equation}
Due to the inductive hypothesis, we have that
\begin{equation}\label{L+2=m0EqB3}
    H_{k + 1} \ > \ 2F_{k - L + 2} + \sum^{\frac{L - 1}{2} - 1}_{i = 0}F_{k - 2i + 1},
\end{equation}
and then using the Fibonacci recurrence and \eqref{L+2=m0EqB1}, we proceed as follows:
\begin{eqnarray}\label{L+2=m0EqB4}
\begin{aligned}
H_{k + 1} \ &> \ 2F_{k - L + 2} + F_k + \sum^{\frac{L - 1}{2} - 1}_{i = 1}F_{k - 2i + 1} \\
\ &> \ F_{k + 1} + F_k
\ = \ F_{k + 2}.
\end{aligned}
\end{eqnarray}
This completes the proof when $L$ is odd.
\end{proof}

As a corollary of the proof of Lemma \ref{L+2=m0}, we see that if $m_0 \geq L + 2$, then either $m_0 = L + 2$ (if $L$ is odd) or $H_m = F_m$ for all $m \in \N^+$ (if $L$ is even). In addition, this lemma shows that there are PLRS of order other than $2$ that generate the Fibonacci numbers, an observation that coincides with the result of Example \ref{PLRSOrder4CEX}.

We may now turn our attention to the scenarios where $m_0 \leq L + 1$.

%%%%%%%%%%
%%%%%%%%%%
%%%%%%%%%%

\subsection{The $m_0 = L + 1$ case}\label{m0=L+1}

The purpose of this lemma is to understand a specific PLRS that serves as an ``edge case" for analyzing asymptotic growth when $L$ is odd.

\begin{lemma}\label{SpecialCOdd3Lem}
Let $k \in \N^+$ and let $\{H_j\}^{\infty}_{j = 1}$ be the PLRS of order $2k + 1$ with $c_i = 1$ for all odd $i \leq 2k - 1$, $c_i = 0$ for all even $i \leq 2k$, and $c_{2k + 1} \geq 3$. Then this sequence has the following properties in comparison to the Fibonacci numbers:
\begin{itemize}
    \item $H_j = F_{j + 1}$ for all $j \leq 2k + 1$;
    \item $H_j > F_{j + 1}$ for all $j > 2k + 1$.
\end{itemize}
\end{lemma}

\begin{proof}
The first claim of this proof can be easily checked by using the definition of a PLRS and the given coefficients for $\{H_j\}^{\infty}_{j = 1}$. The second claim will follow from a strong induction argument. The base case is where $j = 2k + 2$, so we show that $H_{2k + 2} - F_{2k + 3} > 0$. We use the definition of the recurrence and the equalities $H_j = F_{j + 1}$ for all $1 \leq j \leq 2k + 1$ to calculate
\begin{multline}\label{SpecialCOdd3LemEq1}
    H_{2k + 2} - F_{2k + 3} \ = \ \sum^{k}_{i = 1}H_{2i + 1} + 3H_1 - F_{2k + 3} \ \geq \ \sum^{k - 1}_{i = 1}H_{2i + 1} + 3 - F_{2k + 1} \ = \ \dots \\
    \ = \ H_3 + 3 - F_5 \ = \ 3 + 3 - 5 \ = \ 1 \ > \ 0.
\end{multline}
Now we may proceed to the inductive step. Supposing that $H_j > F_{j + 1}$ for all $2k + 2 \leq j \leq m$, we want to show that $H_{m + 1} > F_{m + 2}$. Once again, we may use the recurrence defining $\{H_j\}^{\infty}_{j = 1}$, along with the inductive hypothesis, as follows:
\begin{multline}\label{SpecialCOdd3LemEq2}
    H_{m + 1} - F_{m + 2} \ = \ \sum^{k}_{i = 1}H_{m - 2k + 2i} + c_{2k + 1}H_{m - 2k} - F_{m + 1} - F_m \ > \ \sum^{k - 1}_{i = 1}H_{m - 2k + 2i} + 3H_{m - 2k} - F_m \ \\ \geq \ \sum^{k - 2}_{i = 1}H_{m - 2k + 2i} + 3H_{m - 2k} - F_{m - 2} \ \geq \ \dots \ \geq \ 3H_{m - 2k} - F_{m - 2k + 1} - F_{m - 2k} \ \geq \ H_{m - 2k} \ > \ 0,
\end{multline}
which completes the proof.
\end{proof}

Now we are able to fully dissect the case $m_0 = L + 1$, when $L$ is either odd or even.

\begin{lemma}\label{L+1=m0}
If $m_0 = L + 1$, then the following behaviors take place:
\begin{itemize}
    \item If $L$ is even and we choose $c_L \geq 2$, then $H_m > F_{m + 1}$ for all $m \geq m_0$.
    \item If $L$ is odd and we choose $c_L = 1$, then $H_m < F_{m + 1}$ for all $m \geq m_0$
    \item If $L$ is odd and we choose $c_L \geq 3$ then $H_m > F_{m + 1}$ for all $m \geq m_0$
\end{itemize}
\end{lemma}

\begin{proof}
First, consider what happens if $L$ is even. From the condition $m_0 = L + 1$, we must have that $H_k = F_{k + 1}$ for all $k \leq L$. By using the definition of the recurrence for the PLRS $\{H_j\}^{\infty}_{j = 1}$, it becomes clear that we must have $c_i = 1$ for all $i \leq L - 1$ that are odd, and $c_i = 0$ for all $i \leq L - 1$ that are even. Since the PLRS has order $L$, we must have $c_L \geq 1$. If $c_L = 1$, then $H_m = F_{m + 1}$, which contradicts the assumption that $m_0 = L + 1$. Thus we must have that $c_L \geq 2$. Considering the event where $c_L = 2$, we can prove that $H_m > F_{m + 1}$ for all $m \geq m_0 = L + 1$ using strong induction. One can easily check the base case, that $H_{m_0} > F_{m_0 + 1}$. Now, assume that $H_j > F_{j + 1}$ holds for all $L + 1 \leq j \leq k$, and we want to show that $H_{k + 1} > F_{k + 2}$. This proof follows from utilizing the recurrence defining $\{H_j\}^{\infty}_{j = 1}$, and the fact that $H_j \geq F_{j + 1}$ for all $j \leq k$:
\begin{multline}\label{L+1=m0EqA1}
    H_{k + 1} - F_{k + 2} \ = \ \sum^{\frac{L}{2}}_{i = 1}H_{k - L + 2i} + 2H_{k - L + 1} - F_{k + 1} - F_{k + 2} \\ > \ \sum^{\frac{L}{2} - 1}_{i = 1}H_{k - L + 2i} + 2H_{k - L + 1} - F_{k - 1} - F_{k - 2} \ \geq \ \dots \ 
    \ \geq \ 2H_{k - L + 1} - F_{k - L + 2} - F_{k - L + 1} \ \geq \ 0,
\end{multline}
which completes the proof in this case. Among other things, this shows that if $L$ is even and $m_0 = L + 1$, the PLRS $\{H_j\}^{\infty}_{j = 1}$ will have asymptotic growth at least equal to the one of the Fibonacci numbers.

Now we consider what happens if $L$ is odd. Similarly to the proof of Lemma \ref{L+2=m0}, the values of most of the coefficients will be fixed since we need $F_m = H_m$ for all $m \leq L$. We require that $c_i = 1$ for all odd $i$ less than $L$, and $c_i = 0$ for all even $i$; the value of $c_L$ must be positive since the PLRS has order $L$, but its value is flexible. If $c_L = 2$ then we end up in the scenario described in Lemma \ref{L+2=m0}, so now we consider other possible values of $c_L$. If $c_L = 1$ then the recurrence is
\begin{equation}\label{L+1=m0EqB1}
    H_k \ = \ \sum^{\frac{L - 1}{2}}_{i = 0}H_{k - 2i - 1},
\end{equation}
and one can easily check that $H_{m_0} < F_{m_0 + 1}$. We can then prove that $H_m < F_{m + 1}$ for all $m > m_0$ via strong induction. To do so, we assume that $H_j < F_{j + 1}$ holds for some $L + 1 \leq j \leq k$, and prove that $H_{k + 1} < F_{k + 2}$. We apply the recursion \eqref{L+1=m0EqB1} replacing $k$ with $k + 1$, and then use the inductive hypothesis to obtain
\begin{equation}\label{L+1=m0EqB2}
    H_{k + 1} \ = \ \sum^{\frac{L - 1}{2}}_{i = 0}H_{k - 2i} \ < \ F_{k + 1} + \sum^{\frac{L - 1}{2}}_{i = 1}H_{k - 2i} \ \leq \ F_{k + 1} + \sum^{\frac{L - 1}{2}}_{i = 1}F_{k - 2i + 1}.
\end{equation}
How we proceed from here changes slightly depending on the parity of $k$. If $k$ is even then we can complete the induction proof via the following argument:
\begin{eqnarray}\label{L+1=m0EqB3}
\begin{aligned}
    &F_{k + 1} + \sum^{\frac{L - 1}{2}}_{i = 1}F_{k - 2i + 1} \ \leq \ F_{k + 1} + \sum^{\frac{k - 2}{2}}_{i = 1}F_{k - 2i + 1} + F_2 \\
    \ &= \ F_{k + 1} + \sum^{\frac{k - 4}{2}}_{i = 2}F_{k - 2i + 1} + F_2 
    \ = \ \dots 
    \ = \ F_{k + 1} + F_k \ = \ F_{k + 2}.
\end{aligned}
\end{eqnarray}
Meanwhile, if $k$ is odd, then $k \geq L + 1$ implies that $k \geq L + 2$, and we instead proceed this way:
\begin{eqnarray}\label{L+1=m0EqB4}
\begin{aligned}
    &F_{k + 1} + \sum^{\frac{L - 1}{2}}_{i = 1}F_{k - 2i + 1} \ < \ F_{k + 1} + \sum^{\frac{k - 3}{2}}_{i = 1}F_{k - 2i + 1} + F_3 \\
    \ &= \ F_{k + 1} + \sum^{\frac{k - 3}{2}}_{i = 2}F_{k - 2i + 1} + F_5 
    \ = \ \dots 
    \ = \ F_{k + 1} + F_k \ = \ F_{k + 2}.
\end{aligned}
\end{eqnarray}
On the other hand, if $c_L \geq 3$, Lemma \ref{SpecialCOdd3Lem} identifies the specific behavior that this sequence takes: namely, we will have $H_k = F_{k + 1}$ for all $k \leq L + 1$, and then $H_k > F_{k + 1}$ for all $k \geq L + 2$, which is exactly what we wanted to show. 
\end{proof}

%%%%%%%%%
%%%%%%%%%
%%%%%%%%%

\subsection{The $m_0 \leq L$ case}\label{m0Small}
The case $m_0 \leq L$ introduces substantial difficulties not present in the other cases, and we are mostly leaving this case open in this paper. The main issue is that Proposition \ref{uniquenessDetPLRS} no longer applies. Indeed, there is no longer enough information to uniquely determine all coefficients of the PLRS. It turns out the coefficients that are not uniquely determined by the first few values of the PLRS can have their sizes manipulated to change the sequence's growth rate. For instance, if $\{H_j\}^{\infty}_{j = 1}$ is a PLRS of order $L$ for which the coefficients $c_1, c_2, \dots, c_{L - 1}$ are fixed, then there exists a possible value of $c_L \in \N^+$ for which $\lim_{m \rightarrow \infty}\frac{H_{m + 1}}{H_m} > \phi$. This follows from the fact that the constant term of \eqref{GenSetupEq2} is $-c_L$. As $c_L$ becomes larger, the unique positive root of the growth polynomial can be made arbitrarily large.
 
%%%%%%%%%
%%%%%%%%%
%%%%%%%%%

\section{Concluding remarks}\label{end}

We conclude this paper by discussing some potential directions of further, future investigations. The first direction is to consider asymptotics of compositions where only a sub-sequence of the Fibonacci numbers is allowed, for instance only the even-indexed Fibonacci numbers. This could serve as a more combinatorial companion to the recent paper \cite{Gil}, which provides conditions on which decompositions of natural numbers into sub-sequences of Fibonacci numbers exist and are unique. Such a generalization is not necessarily encompassed by our theory of PLRS, because the first term may no longer be equal to $1$.

Another possibility would be to extend our PLRS theory to recurrences where the coefficients need not to be integers. This would give us a richer class of sequences to compare to each other in our asymptotic framework. Notably, the theory in section \ref{PLRS} does not specifically use the property that PLRS are integer sequences, but only that certain coefficients of the recurrence defining the PLRS were positive. The trade-off, however, is that we would no longer be studying integer composition problems.

Finally, Section \ref{PLRS} of this paper left partially open the characterization of finding which PLRS sequences have greater asymptotic growth than the Fibonacci numbers, and which have smaller such growth. While some results exist to bound the modulus of roots of growth polynomials from above and below, we have not identified a complete characterization of which growth polynomials fall into which category.

%%%%%%%%%
%%%%%%%%%
%%%%%%%%%

\vspace{2 cm}

% \noindent {\bf Acknowledgement.} The author would like to thank the editorial staff at {\it Integers} for providing extensive, detailed feedback on the manuscript.
 
 %%%%%%%%%
%%%%%%%%%
%%%%%%%%%

\end{document}